\documentclass[a4paper,11pt]{amsart}
\usepackage{amsmath,amssymb,amscd,amsfonts, amsthm}
\usepackage{amssymb}
\usepackage[T1]{fontenc}
\usepackage{amsrefs, esint}
\numberwithin{equation}{section}
\newtheorem{theorem}{Theorem}[section]

\newtheorem{rem}[theorem]{Remark}
\newtheorem{defn}[theorem]{Definition}
\newtheorem{lemma}[theorem]{Lemma}

\newtheorem{prop}[theorem]{Proposition}
\newtheorem{conjecture}[theorem]{Conjecture}

\newtheorem{example}[theorem]{Example}

\def\det{\mathop{\rm det}\nolimits}

\def\d{\partial}

\def\cR{{\mathcal R}}
\def\cE{{\mathcal E}}

\def\cC{{\mathcal C}}

\let\ep=\varepsilon

\def\bC{{\mathbb C}}
\def\bR{{\mathbb R}}

\def\bD{{\mathbb D}}

\def\z{{\zeta}}

\title[ PSH Symmetrizations]
{The Lelong number, the Monge-Amp\`ere mass and  the Schwarz symmetrization of plurisubharmonic functions}

\author{Long LI}
\address{Science Institute, University of Iceland, Reykjavik, Iceland.}
\email{longli@hi.is}

\begin{document}
\maketitle 

\begin{abstract}
The aim of this paper is to study the Lelong number, the integrability index and the  Monge-Amp\`ere mass at the origin 
of an $S^1$-invariant plurisubharmonic function on a balanced domain in $\bC^n$ under the Schwarz symmetrization. 
We prove that $n$ times the integrability index is exactly 
the Lelong number of the symmetrization,
and if the function is further toric with a single pole at the origin, 
then the Monge-Amp\`ere mass is always decreasing under the symmetrization. 
\end{abstract}

\section{Introduction}
Let $\Omega\subset \bC^n$ be a bounded domain containing the origin $\underline{O} $, 
and $u$ be a plurisubharmonic function defined on $\Omega$. 
Assume that the pluri-polar set $\{ u = -\infty \}$ is non-empty in $\Omega$. 
Then we are interested to study the singularity of $u$ at the origin.
In general, there are three useful quantities to characterise this singularity. 

First, the Lelong number of $u$ at the origin is defined as 
$$ \nu_u(0) = \liminf_{z\rightarrow 0} \frac{ u(z) }{\log |z|};$$
this is the supreme of all numbers $\gamma \geq 0$ such that 
$$ u(z) \leq \gamma \log |z| + O(1) $$
near the origin. Moreover, one can show 
$$ \nu_{u}(0) = \lim_{r\rightarrow 0}  \frac{1} { a_{2n-2}r^{2n-2} }\int_{B_r} \frac{\Delta u}{2\pi} , $$
where $\frac{1}{2\pi} \Delta u$ is the Riesz measure of $u$, and  $a_N$ is the volume of the unit ball in $\bR^N$.

The second quantity is the integrability index of $u$ at the origin,
and it is defined as 
$$ \iota_u(0) = \inf\{ r >0; \ \ \  e^{-\frac{2u}{r}} \in L^1_{loc}(\underline{O})    \}. $$
If we assume that $u$ is not identically equal to $-\infty$ near the origin, 
then $\iota_u(0)$ will take its value in $[0, +\infty)$. 
According to Demailly and Koll\'ar \cite{DK}, 
the inverse of $\iota_u(0)$ is named as the \emph{complex singularity exponent} of $u$ at the origin, 
and the following sharp estimate is obtained from Skoda's work \cite{Sko}
\begin{equation}
\label{intro-000}
 \frac{1}{n} \nu_u(0) \leq \iota_u(0) \leq  \nu_u(0). 
\end{equation}

The third quantity is the \emph{residue Monge-Amp\`ere mass} of $u$ at the origin defined as 
\begin{equation}
\label{intro-001}
 \tau_{u}(0) = (dd^c u)^n|_{\{ \underline{O} \}}, 
 \end{equation}
whenever the RHS of equation (\ref{intro-001}) is well defined. 
There are many cases in which this residue mass can not make any sense.  
However, it was shown that the Monge-Amp\`ere measure $(dd^c u)^n$ 
is always well defined provided that the polar set of $u$  is contained in a compact subset $K\subset\Omega$ \cite{Dem3}.

There are many beautiful works to describe properties of these quantities or relation between them. 
The purpose of this paper is to study how these quantities change under certain symmetrization process, 
when the plurisubharmonic function is also $S^1$-invariant.  

At this moment, we identify $\bC^n$ as $\bR^{2n}$,
and then $\Omega$ is a bounded, open and connected set in the real space.
The \emph{Schwarz symmetrization} of a real valued measurable function $u$ on $\Omega$
is a radial function $\hat u(x) = f(|x|)$, 
with $f$ non-decreasing and equimeasurable with $u$. 
That is to say, for each $t\in \bR$ we have 
$$ |\{ u <t \}| = |\{\hat u <t \}|.  $$

Back to the complex setting,
one can ask the question whether the Schwarz symmetrization of a plurisubharmonic function  is still plurisubharmonic. 
Unfortunately, this is not always the case, and any general Green kernel on the unit disk will do a counter-example \cite{BB}. 
However, Berman and Berndtsson (Theorem (2.3), \cite{BB}) confirm this question when the plurisubharmonic function is also $S^1$-invariant.

Assume further that  $\Omega $ is a \emph{balanced domain} in $\bC^n$.
Consider the following $S^1$-action for any point $z\in \Omega$ as 
$$z \rightarrow e^{i\theta}z = (e^{i\theta}z_1, \cdots, e^{i\theta}z_n),  $$
for all $\theta\in\bR$. 
Then a function $f(z)$ is called \emph{$S^1$-invariant } if $f(e^{i\theta} z) = f(z)$ for every $z\in \Omega$ and all $\theta\in\bR$.

Based on Berman-Berndtsson's result, our results are presented as follows. 
For simplicity, the domain $\Omega$ will always be taken as the unit ball $B\subset\bC^n$ in the statement. 

\begin{theorem}
\label{intro-main}
Let $u$ be an $S^1$-invariant plurisubharmonic function on the unit ball $B$, 
which can be extended invariantly to a slightly larger ball $B_{1+\delta}$. Let $\hat u$ be its Schwarz symmetrization. 
Then its Lelong number and integrability index both reach their maximums at the origin, i.e. 
we have 
\begin{equation}
\label{intro-002}
\nu_{u}(0) = \max_{x\in B} \nu_u(x); \ \ \ 
\iota_u(0) = \max_{x\in B} \iota_u(x).
\end{equation}
In particular, the following formula holds: 
\begin{equation}
\label{intro-003}
\iota_u(0) = \frac{\nu_{\hat u} (0)}{n} = \lim_{t\rightarrow -\infty} \frac{2t}{\log \left|    \{ u<t \}  \right|}.
\end{equation}
\end{theorem}

This main result will be proved through Proposition (\ref{prop-S1-002}), Theorem (\ref{thm-main}) and Theorem (\ref{thm-main2}) in later sections. 
The key observation is a simple fact. 
Let $\mathfrak{l}_z = \{sz \}, s\in[0,1]$ be the line segment 
connecting the origin  and a point $z\in B$.
Then $u$ must be non-decreasing along this line segment,
since its restriction on the holomorphic disk $D_z = \{ \lambda z\}, \lambda\in \bD$ is a radial, subharmonic function \cite{BB}.  

Notice that the symmetrization $\hat u: B \rightarrow \bR$ is a radial, non-decreasing plurisubharmonic function, 
with a single pole at the origin. In this case, 
it is well known (Proposition (\ref{prop-app-001}), Appendix) that the residue Monge-Amp\`ere mass is exactly  the $n$th-power of the Lelong number
at the origin, i.e. we have 
\begin{equation}
\label{intro-004}
\tau_{\hat u} (0) = [\nu_{\hat u} (0)]^n. 
\end{equation}
In particular, we have $\tau_{\hat u} (0) = 0$ if $\nu_u (0) =0$ from equation (\ref{intro-000}). 

More generally, the residue mass  of plurisubharmonic functions with toric symmetry was studied by Rashkoskii \cite{Rash1}. 
That is to say, $u(z)$ is invariant under the following $(S^1)^{\times n}$-action on a balanced Reinhardt domain $\Omega$ 
$$ z\rightarrow (e^{i\theta_1}z_1, \cdots, e^{i\theta_n}z_n), $$
for all $\theta_j \in \bR$ and $j =1,\cdots, n$. 

Furthermore, Rashkoskii \cite{Rash} also found 
a lower bound of the residue mass,  in terms of the so called \emph{refined Lelong numbers}
for plurisubharmonic functions with a single pole at the origin. 

For any vector $a\in \bR_+^n$,
the refined Lelong number of $u$ at the origin, introduced by Kiselman \cite{Kis3}, is 
defined as 
\begin{eqnarray}
\label{intro-0040}
\nu_u(0, a)&=& \lim_{t\rightarrow -\infty} t^{-1} \sup\{ u(z);\ \ |z_k|\leq e^{a_k t},\ 1\leq k \leq n  \}
\nonumber\\
&=& \lim_{t\rightarrow -\infty} t^{-1} T_u(0, ra),
\end{eqnarray}
where $T_u(0, b)$ is the mean value of $u$ over the set $\{ z;\ \ |z_k | = e^{b_k}, \ 1\leq k \leq n \}$ for any $b\in \bR_+^n$.

Based on these results, our identity (equation (\ref{intro-003}))
implies the following domination phenomenon of residue masses under the symmetrization.  

\begin{theorem}[Theorem (\ref{thm-main2})]
\label{intro-main2}
Let $u$ be a toric plurisubharmonic function on the unit ball $B$ with a single pole at the origin, 
which can be extended invariantly to a slightly larger ball $B_{1+\delta}$. Let $\hat u$ be its Schwarz symmetrization. 
Then we have 
$$ \tau_{\hat u} (0) \leq \tau_u(0). $$
\end{theorem}

Several examples for toric plurisubharmonic functions are presented in the last section. 
One can see that the residue Monge-Amp\`ere mass at the origin 
is always decreasing under the symmetrization, 
whenever it is well defined. More interestingly, 
even if it is not well defined as in Kiselman or Cegrell's examples (\cite{Kis1}, \cite{Ceg2}),
we can still compute the residue mass after taking the symmetrization. 

Finally, one conjecture is made, and we expect that this domination phenomenon 
also occurs for all $S^1$-invariant plurisubharmonic functions. 

\bigskip

\textbf{Acknowledgement}:
The author is very grateful to Prof. Xiuxiong Chen and Prof. P\u aun for their continuous encouragement 
in mathematics, and he would like to thank Prof. Demailly, Prof. Sigurðsson and Prof. Lewandowski  for many useful discussions. 
Moreover, he also wants to thank Prof. Rashkovskii who gave many helpful comments on the draft of this paper.

\section{Preliminaries }

\subsection{The increasing rearrangement}
Let $E$ be a (Lebesgue) measurable subset of $\mathbb{R}^N$, and we denote its $N$-dimensional (Lebesgue) measure by $|E|$. 

Suppose $\Omega \subset \mathbb{R}^N $ is a bounded measurable set. Let $u: \Omega \rightarrow \mathbb{R}$ be a measurable function.
For any $t\in \mathbb{R}$, the sub-level set of $u$ is defined as 
$$ \{ u < t \}: = \{ x\in \Omega; \ u(x) <  t \}.$$
Then the \emph{distribution function} of $u$ is given by 
$$ \mu(t, u ) = |\{ u<t \}|.$$
This function is a monotonically increasing function of $t$, and 
for $t\leq \text{ess. inf} (u)$, we have $\mu(t,u) = 0$, while for $t\geq \text{ess. sup}(u)$, we have $\mu(t,u) = |\Omega|$. 

The \emph{increasing rearrangement} of $u$ is a function, denoted $u_*$, is defined on $[0, |\Omega| ]$ by 
$$u_* (|\Omega|) = \text{ess. sup} (u)$$ 
\begin{equation}
\label{P-001}
u_* (s) = \inf \{t\in\mathbb{R}; \  |\{ u<t \}| >s   \}, \ \ \ \ \  0\leq s < |\Omega|.
\end{equation}

This new function $u_*$ is essentially the inverse function of $\mu(t,u)$,
but it is always non-decreasing and right-continuous. 
In fact, the distribution function $\mu(t,u)$ is strictly increasing for a continuous function $u$, 
and then $u_*$ must also be continuous. 

Moreover, the mapping $u\rightarrow u_*$ is non-decreasing, i.e. 
if $u\leq v$, where $u$ and $v$ are real valued function on $\Omega$, then $u_* \leq v_*$.

\begin{defn}
\label{defn-P001}
Two real valued functions (with possibly different domains of definition) 
are said to be equimeasurable if they have the same distribution functions.
\end{defn}

One important figure of the increasing rearrangement is that two functions $u: \Omega \rightarrow \mathbb{R}$ 
and $u_*: [0, |\Omega|]\rightarrow \mathbb{R}$ are \emph{equimeasurable}, i.e. 
we have 
\begin{equation}
\label{P-002}
|\{ u < t \}| = |\{ u_* < t\}|,
\end{equation}
for all $t\in\mathbb{R}$.
More generally,  the following facts are well known, and readers 
can refer to Kesavan's book \cite{Kes}.
\begin{lemma}
\label{lem-P001}
Let $u:\Omega\rightarrow \mathbb{R}$ be measurable. Let $F: \mathbb{R}\rightarrow \bR$
be a non-negative Borel measurable function. Then 
\begin{equation}
\label{P-003}
\int_{\Omega} F(u(x))dx = \int_0^{|\Omega|} F(u_* (s))ds.
\end{equation}
\end{lemma}

\begin{lemma}
\label{lem-P002}
Let $\Omega\subset\bR^N$ be bounded and $u:\Omega\rightarrow \mathbb{R}$ be an integrable function. 
Let $E\subset \Omega$ be a measurable subset.
Then 
\begin{equation}
\label{P-004}
\int_{E} u(x)dx \geq \int_0^{|E|} u_* (s)ds.
\end{equation}
Equality holds in equation (\ref{P-004}) if and only if,
\begin{equation}
\label{P-005}
(u|_{E})_* = u_*|_{\{[ 0, |E ] \}}, \ \ \emph{a.e.}
\end{equation}
\end{lemma}
Although we are not going to use, but it is still worthy  mentioning 
that the equation (\ref{P-005}) holds if $E$ is exactly a sub-level set of $u$, i.e. we have 
$$ \int_{\{ u <t \}} u(x)dx = \int_0^{|\{u<t\}|} u_* (s)ds.$$

\subsection{The Schwarz symmetrization}
Given a measurable subset $E$ in $\bR^N$ of finite measure, 
we will denote by $\hat{E}$, the open ball centred at the origin $\underline{O}$ 
and having the same measure as $E$, i.e. $|E| = |\hat{E}|$. 
Let $a_N$ be the volume of the unit ball in $\bR^N$. That is to say 
$$ a_N = \frac{\pi^{\frac{N}{2}}}{\Gamma(\frac{N}{2}+1)},$$
where $\Gamma(s)$ is the gamma function. 

\begin{defn}
\label{defn-P002}
Let $\Omega\subset \bR^N$ be a bounded domain, and $u:\Omega\rightarrow \bR$ be 
a measurable function. Then its Schwarz symmetrization is the function $\hat{u}: \hat\Omega \rightarrow \bR$ defined by 
$$ \hat{u} (x) = u_*(a_N |x|^N),\ \ \ x\in \hat\Omega. $$
\end{defn}

Taking $|x| = r$ and $s = a_N r^N$, we have the following from the  change of variables:
$$\int_{\hat\Omega} \hat{u}(x) dx = \int_0^{|\Omega|}  u_*(s)ds. $$

Several useful properties of the Schwarz symmetrization are listed in the following Proposition.

\begin{prop}
\label{prop-P001}
Let $\Omega\subset \bR^N$ be a bounded domain, and $u:\Omega\rightarrow \bR$ be 
a measurable function.
Let $\hat u: \hat\Omega \rightarrow \bR$ be its Schwarz symmetrization. Then we have
\begin{enumerate}
\item[(i)] 
$\hat{u}$ is radially symmetric and non-decreasing.

\item[(ii)] 
$u$, $u_*$ and $\hat{u}$ are all equimeasurable. 

\item[(iii)] 
If $F: \bR \rightarrow \bR$ is a non-negative Borel measurable function, then 
$$\int_{\hat\Omega} F(\hat{u}(x)) dx  = \int_{\Omega} F(u(x)) dx. $$

\item[(iv)]
If $G : \bR \rightarrow \bR$ is a non-decreasing function, then 
$$ \widehat{G(u)} = G(\hat u), \ \ \ \emph{a.e.}$$

\item[(v)]
If $E\subset \Omega$ is a measurable subset, then 
$$ \int_E u(x) dx \geq \int_0^{|E|} u_*(s)ds = \int_{\hat E} \hat u(x)dx. $$
Equality occurs if and only if, $\widehat{( u|_{E} )} = \hat u|_{\hat E}$.

\item[(vi)]
(P\'olya-Szeg\"o) Let $1\leq p < \infty$. 
Let $u\in W^{1,p}_0(\Omega)$ be a non-positive function. Then we have $\hat u\in W^{1,p}_0(\hat\Omega)$ and 
$$ \int_{\hat\Omega} |\nabla \hat u|^p dx \leq \int_{\Omega} |\nabla u|^p dx. $$
\end{enumerate}
\end{prop}

\section{$S^1$-invariant plurisubharmonic functions}

Let $\Omega$ be an open, connected and bounded subset  of $\bR^N$ with $N=2n, n\in \mathbb{Z}^+$.  
Then the set $\Omega$ can also be viewed as a domain in $\bC^n$. 
It is called a \emph{balanced domain} if for every $\lambda \in \bD$ (the unit disk) and $z\in \Omega$, we have $\lambda z\in \Omega$. 
Consider the following $S^1$-action on $\Omega$: 
$$z \rightarrow e^{i\theta}z = (e^{i\theta}z_1, \cdots, e^{i\theta } z_n), $$
for $\theta\in \bR$. 
A function $f$ defined on a balanced domain is called \emph{ $S^1$-invariant } if $f(e^{i\theta}z) = f(z)$ for all $\theta\in\bR$ at every $z\in\Omega$.

Let $u$ be a plurisubharmonic function on a balanced domain $\Omega$. 
It is natural to ask whether its Schwarz symmetrization $\hat u$ on $\hat\Omega$ is still plurisubharmonic. 
Unfortunately, this is not true in general. As indicated in Berman-Berndtsson \cite{BB}, the Green function 
on the complex disk 
$$ u(z) = \log \left|  \frac{z-a}{1-\bar a z}\right| $$
has plurisubharmonic Schwarz symmetrization $\hat u$ only if $a=0$.

However, Berman-Berndtsson \cite{BB} gave an affirmative answer to this question when $u$ is $S^1$-invariant.
\begin{theorem}[Berman-Berndtsson]
\label{thm-S1-001}
Let $\Omega$ be a balanced domain in $\bC^n$, and $u$ be an $S^1$-invariant plurisubharmonic function on $\Omega$.
Then its Schwarz symmetrization $\hat u$ is plurisubharmonic on $\hat\Omega$. 
\end{theorem}

Since we are interested with the singularity of $u$, we assume that its polar set $\{ u =-\infty \}$ is non-empty in $\Omega$.
By the maximum principle. we can also assume that $\sup_{\Omega} u = \sup_{\d\Omega} u = 0$. 
Then its symmetrization function  $\hat u: \hat\Omega \rightarrow \bR$ is radially symmetric, non-decreasing w.r.t. the radius $r=|z|$,
and reaches its maximum on the boundary, i.e. $\sup_{\hat\Omega} \hat u = \hat u|_{\d\hat\Omega} =0  $.

Moreover, $\hat u$ is continuous outside the origin and decreases to $-\infty$ as $r$ is converging to zero,
since the function $ f( t ): = \hat u ( e^t ) $ is convex and bounded from above 
for  $t\in (-\infty, 0)$ by Berman-Berndtsson's result. 

The convex function $f(t)$ is locally Lipschitz, and then its first derivative $f'(t)$ exists almost everywhere  
and is non-decreasing for $t\in (-\infty, 0)$.
In fact, the following limit 
$$ \lim_{t\rightarrow -\infty }f'(t)$$ 
always exits and is equal to the Lelong number $\nu_{\hat u} (0)$ of $\hat u$ at the origin.

Then we can compare this Lelong number $\nu_{\hat u}(0)$ with the original one $\nu_u(0)$. 
Notice that the Lelong number of a plurisubharmonic function is purely a local concept. 
Hence we will assume that the domain $\Omega$ is the unit ball $B \subset\bC^n$ from now on. 

Let  $u$ be an $S^1$-invariant plurisubharmonic function on $B$,
and we say that it can be \emph{extended invariantly} to a larger ball $B_{1+\delta}$,
if there exists an $S^1$-invariant plurisubharmonic function $v$ on $B_{1+\delta}$ such that the restriction $v|_{B}$ is equal to $u$.
Based on these assumptions, we state our main theorem as follows. 

\begin{theorem}
\label{thm-main}
Let $u$ be an $S^1$-invariant plurisubharmonic function on the unit ball $B$, 
which can be extended invariantly to a slightly larger ball $B_{1+\delta}$. Let $\hat u$ be its Schwarz symmetrization. 
Then we have 
\begin{equation}
\label{S1-001}
\nu_u(0) \leq \nu_{\hat u} (0) \leq n \nu_{u}(0). 
\end{equation}
In particular, if $\nu_u (0) =0$, then $\nu_{\hat u}(0) = 0$.
\end{theorem}

The first observation is that the Schwarz symmetrization always increases the Lelong number at the origin. 

\begin{lemma}
\label{lem-S1-001}
Let $u$ be an $S^1$-invariant plurisubharmonic function on the unit ball $B$, 
and $\hat u$ be its Schwarz symmetrization. 
Then we have 
$$\nu_u(0) \leq \nu_{\hat u} (0).$$ 
\end{lemma}
\begin{proof}
Take the following average of $u$ on a small ball $B_r$ centred at the origin: 
$$ V_u(0,r) = \frac{1}{a_{2n} r^{2n}} \int_{\bar B_r} u d\lambda. $$
Then the Lelong number $\nu_u (0)$ is equal to the limit 
$$ \lim_{r\rightarrow 0} \frac{V_u(0,r)}{\log r}. $$
However, a basic property of the symmetrization, Proposition \ref{prop-P001} - (v), says that we have 
$$\int_{\bar B_r} u d\lambda \geq \int_{\bar B_r} \hat u d\lambda,$$
since $\widehat {(\bar B_r)} = \bar B_r$. This implies $ V_{u} (0,r) \geq V_{\hat u} (0,r)$ for each $r$ small,
and we conclude the proof by taking $r\rightarrow 0$ as 
\begin{equation}
\label{S1-002}
 \nu_u (0) = \lim_{r\rightarrow 0} \frac{V_u(0,r)}{\log r} \leq   \lim_{r\rightarrow 0} \frac{V_{\hat u }(0,r)}{\log r} = \nu_{\hat u}(0). 
 \end{equation}
\end{proof}
Before going to the proof of the reversed inequality, we need to introduce the following tool, which is studied by Demailly and Koll\'ar  (\cite{DK}). 

\subsection{The complex singularity exponent}

Let $u$ be a plurisubharmonic function on a domain $\Omega$ in $\bC^n$.
For any point $x\in\Omega$, we introduce the \emph{complex singularity exponent}
of $u$ at $x$ as 
$$ \mathcal{C}_u (x) : = \sup \{ c\geq 0;\ \  e^{-2c u} \  \emph{is $L^1$ on a neighbourhood of $x$ }  \}. $$

This number $\cC_u(x)$ will take its value in $(0, +\infty]$, if we assume that 
$u$ is not identically equal to $-\infty$ in a neighbourhood $x$.
By equation (\ref{intro-000}), we further have the following estimate  
\begin{equation}
\label{S1-003}
n^{-1} \nu_u (x) \leq \cC^{-1}_u (x) \leq \nu_u(x ),
\end{equation}
where $\nu_u(x)$ is the Lelong number of $u$ at $x$.
More generally, we can define the complex singularity exponent of $u$ on any relatively compact sub-domain $\Omega'\subset\subset \Omega$ as
$$  \cC_{u}(\Omega') : = \sup \{ c\geq 0;\ \  e^{-2c u} \  \emph{is $L^1$ on $\Omega'$ }  \}. $$
It is clear that for any $x\in \Omega'$ we have
$$\cC_{u }(\Omega') \leq \cC_u(x).$$
Then we are going to prove a simpler version of Theorem (\ref{thm-main}) first.

\begin{prop}
\label{prop-S1-001}
Let $u$ be an $S^1$-invariant plurisubharmonic function on the unit ball $B$, 
which can be extended invariantly to a slightly larger ball $B_{1+\delta}$.
Assume that the Lelong number of $u$ on the closed unit ball reaches its maximum at the origin, i.e. 
$$ \sup_{x\in \bar B} \nu_u(x) =  \nu_u (0).$$
Then we have $ \nu_{\hat u}(0) \leq n \nu_u (0)$. 
\end{prop}
\begin{proof}
As explained before, we can assume that the function $u$ is always  negative and has non-trivial polar set, 
and then the symmetrization $\hat u$ on $B$ is also negative, radially symmetric, non-decreasing 
and has only a single pole at the origin $\underline{O}$. Then it is clear that we have 
$$\cC_{\hat u}(\underline{O} ) = \cC_{\hat u} (B), $$
and then the inequality $\nu_{\hat u} (0) \leq n \cC_{\hat u}^{-1}(B)$ follows from equation (\ref{intro-000}). 
On the other hand, we claim that the following estimate holds: 
\begin{equation}
\label{S1-0040}
\nu_u^{-1} (0) \leq \cC_{u} (B). 
\end{equation}
In fact, we have 
$$\nu_u^{-1}(0) \leq \nu_u^{-1} (x) \leq \cC_u (x)$$
for all $x\in\bar B$ by our assumptions and equation (\ref{intro-000}). 
Taking any real number $ 0< c <\nu_u^{-1} (0)  $,
 there exist a small radius $0< r < \delta/10 $ for each $x\in \bar B$ such that 
the following integral is finite 
$$ \int_{B_r (x)} e^{-2c u} d\lambda < +\infty. $$

Moreover, there are finitely many such balls $\{ B_{r_j} (x_j) \}_{j=1, \cdots, k}$ covering the closed unit ball $B$,
and their union is contained in $B_{1+\delta}$. 
Eventually, we can control the following integral as 
\begin{equation}
\label{S1-004}
\int_B e^{-2cu}d\lambda \leq \sum_{j=1}^k \int_{B_{r_j} (x_j)}   e^{-2cu} d\lambda < +\infty.
\end{equation} 
This implies $c \leq \cC_u (B)$, for all $c\in (0, \nu^{-1}_u(0 ) )$,
and our claim (equation (\ref{S1-0040}))  follows by taking the supreme. 

Next notice that the complex singularity exponent is unchanged under the symmetrization, i.e. $\cC_u (B) = \cC_{\hat u} (B)$.
This is because we have 
\begin{equation}
\label{S1-005}
\int_{B} e^{-2cu} d\lambda = \int_B e^{-2c\hat u} d\lambda,
\end{equation}
for all $c\in \bR^+$ (two sides can possibly both equal to $\infty$), by Proposition (\ref{prop-P001})-(iii). 
Finally, our estimate follows since we have 
\begin{equation}
\label{S1-006}
\nu_{\hat u} (0) \leq n \cC_{\hat u}^{-1}(B) = n \cC_{u}^{-1}(B) \leq  n \nu_{u} (0). 
\end{equation}
\end{proof}

\subsection{The Lelong number}
In the following, we will argue that the Lelong number 
of an $S^1$-invariant function $u$ indeed reaches its maximum at the center of the ball,
and then the proof of Theorem (\ref{thm-main}) boils down to the case in Proposition (\ref{prop-S1-001}).

A useful observation is made by Berman and Berndtsson \cite{BB} 
to argue that   each  sub-level set $\Omega_t = \{ u < t \}$ is a path-connected domain.  
In fact, if we assume that a point $z\in\Omega$ is contained in the sub-level set $\Omega_t$.
Then the holomorphic disk $D_z = \{ \lambda z\}, \lambda\in \mathbb{D}$ is also contained in $\Omega_t$
by the following lemma. 

\begin{lemma}[Berman-Berndtsson]
\label{lem-S1-002}
There exist a non-decreasing function $g: [0, |z|] \rightarrow \bR \cup \{-\infty \}$ such that
 for all $\lambda \in \mathbb{D}$  we have 
$$u  (\lambda z) = g (|\lambda|). $$
In particular, if $z\in \{ u^{-1}(-\infty) \} $, then  $D_z \subset \{u^{-1}(-\infty) \}$. 
\end{lemma}

This is because the restriction $u|_{D_z}$ is an $S^1$-invariant subharmonic function on the disk $D_z$, 
and then everything follows from the maximum principle. 

For a point $z\in \Omega$, we denote 
$\mathfrak{l}_z =\{s\cdot z \}, s\in[0,1]$ by the line segment connecting the origin  $\underline{O}$ and $z$.
The key observation is that 
the function $u$ is  always \textbf{non-decreasing} along the line segment $\mathfrak{l}_z$
by Lemma (\ref{lem-S1-002}). 
Then we will see that the Lelong  number  also inherits this property.

\begin{lemma}
\label{lem-S1-003}
The Lelong number $\nu_u(x)$ is non-increasing along the line segment $\mathfrak{l}_z$ possible except at the origin.  
\end{lemma}
\begin{proof}
It is enough to prove the following. 
For any point $z' = sz, s\in (0,1)$,
we have $\nu_u(z' ) \geq \nu_{u} (z)  $. 
For any small radius $0 < r < r_0$, where we take 
\begin{equation}
\label{r0}
r_0 = \min \left\{ \frac{\delta s |z|}{100}, \frac{ (1-s)s|z| }{100}  \right\},
\end{equation}
the maximum of $u$ on the ball $B_r(z')$ must be obtained on the boundary, i.e. 
there exist a point $\z\in \d B_{r}(z')$ such that we have 
$$ u(\z) = \max_{B_r(z')} u.$$

Next we can think of  the $n$-dimensional complex space $\bC^n$ as the $2n$-dimensional real space $\bR^{2n}$,
by identifying a point $z\in \bC^n$ with a real vector $X_z \in \bR^{2n}$. 
Consider a plane $\mathfrak{p}$ spanned by the two vectors $X_{z'}, X_{\z}$, i.e. 
$$\mathfrak{p}: = \text{span}\{ X_{z'}, X_{\z} \}.  $$
On this plane, the point $z'$ is the centre of the circle $\mathcal{S'} = \d B_{r}(z') \bigcap \mathfrak{p}$,
and we have $\z\in\mathcal{S'}$. Notice that the point $z$ is also in the plane $\mathfrak{p}$.

Let $\mathcal{S} = \d B_R(z) \bigcap \mathfrak{p}$ be another circle centred at $z$ with radius $R = \frac{r}{s}$,
and then the two circles $\mathcal{S'}$ and $\mathcal{S}$ are disjoint by our choices of $r$ and $R$. 
Let  $\mathfrak{r} = \{ tX_{\z} \}, t\in [ 0, +\infty)$ be a ray initiated from the origin passing through the point $\z$. 
It must also intersect with the circle $\mathcal{S}$ by elementary Euclidean geometry. 
Moreover, if $\xi$ is the last intersection point of the ray $\mathfrak{r}$ and the circle $\mathcal{S}$,
then it is clear to have  $| X_{\xi} | \geq |X_{\z}|  $. 
By considering the holomorphic disk $D_{\xi} = \{\lambda \xi\}, \lambda \in \mathbb{D}$,
we  conclude the following estimate by Lemma (\ref{lem-S1-002}): 
\begin{equation}
\label{S1-007} 
u(\z) \leq u( \xi ) \leq \max_{B_{R}(z)} u. 
\end{equation}

Eventually this implies that we have 
$$ \frac{\max_{B_r(z')} u }{ \log r } = \frac{u(\z)}{\log r} \geq \frac{\max_{B_R(z)} u}{ \log R + \log s},$$
for any $r\in (0, r_0)$.
Since $r = sR$ for some fixed  $s$, 
our result follows as

\begin{eqnarray}
\label{S1-008}
 \nu_{u} (z') &=& \lim_{r \rightarrow 0} \frac{\max_{B_r(z')} u }{ \log r }  
 \nonumber\\
&  \geq & \lim_{R\rightarrow 0}   \frac{\max_{B_R(z)} u}{  ( \log R + \log s )} = \lim_{R\rightarrow 0}   \frac{\max_{B_R(z)} u}{ \log R } 
\nonumber\\
& = & \nu_u(z).
\end{eqnarray}

\end{proof}

\begin{prop}
\label{prop-S1-002}
Let $u$ be an $S^1$-invariant plurisubharmonic function on the unit ball $B$, 
which can be extended invariantly to a slightly larger ball $B_{1+\delta}$.
Then its Lelong number $\nu_u (x)$ 
reaches the maximum at the origin. 
\end{prop}
\begin{proof}
Suppose a point $z\in B$ belongs to the polar set of $u$, 
and $u$ has its Lelong number $\nu_u(z) = c$ at this point.
We claim that the punctured disk $D^*_z$
is contained in the set $\{ \nu_u(x) \geq  c\} $. 
Then the whole disk $D_z$ must be contained in the same set,
since the set $\{ \nu_u(x) \geq  c\} $ is an analytic subset of the unit ball 
by Siu's decomposition theorem \cite{Siu}.

In fact, we can consider a circle as
the boundary of the disk $\d D_z  = \{e^{i\theta}z \}, \theta\in\bR$. 
By our previous Lemma (\ref{lem-S1-003}), the claim will be proved if we can prove for all $z' \in \d D_z$.
$$ \nu_{u} (z') = \nu_u (z).$$

Let $w$ be a maximum point of $u$ on a small ball $B_r(z)$ centred at $z$,i.e. 
$$ u(w) = \max_{B_r(z)} u. $$
Then we can assume that the point $w$ appears on the boundary $\d B_r(z)$. 
For any $z' = e^{i\theta} z$, the point $w' = e^{i\theta} z$ is on the boundary $\d B_r(z')$
since we have 
$$ |z - w| = |e^{i\theta}\cdot (z - w)| = |z' - w'|, $$
and we have $u(w)  = u(w')  \leq \max_{B_r(z') } u$.
Hence the Lelong number is decreasing under this $S^1$-action as 
\begin{equation}
\label{S1-008}
\nu_u(z) = \lim_{r\rightarrow 0} \frac{\max_{B_r (z)} u }{\log r} \geq \lim_{r\rightarrow 0} \frac{\max_{B_r (z')} u }{\log r} = \nu_u (z').
\end{equation}
Similarly, we can prove $\nu_u(z) \leq \nu_u(z')$ by considering the reversed $S^1$-action, i.e. $z= e^{-i\theta} z'$,
and our result follows.

\end{proof}

Finally, Theorem (\ref{thm-main}) is proved by 
combining with  Lemma (\ref{lem-S1-001}), Proposition (\ref{prop-S1-002}) and  Proposition (\ref{prop-S1-001}).

\begin{rem}
Besides the unit ball, 
our arguments in Lemma (\ref{lem-S1-003}) and Proposition (\ref{prop-S1-002}) also work on any other balanced domains in $\bC^n$.
Then we can conclude that the Lelong number $\nu_u(x)$ always obtains its maximum at the origin,
for any $S^1$-invariant plurisubharmonic function $u$ defined on a balanced domain $\Omega$.
\end{rem}

\subsection{The sharp estimate}
The inverse of the complex singularity exponent is called the integrability index \cite{Kis} of a plurisubharmonic function $u$ at a point $x$,
and we denote it by 
$$\iota_{u}(x) = \cC^{-1}_u (x). $$

In this subsection, we will show that this integrability index also reaches its maximum at the origin for an $S^1$-invariant plurisubharmonic function,
and there is an explicit formula for the Lelong number of the Schwarz symmetrization at the origin.

\begin{theorem}
\label{thm-main2}
Let $u$ be an $S^1$-invariant plurisubharmonic function on the unit ball $B$, 
which can be extended invariantly to a slightly larger ball $B_{1+\delta}$. Let $\hat u$ be its Schwarz symmetrization. 
Then we have 
\begin{equation}
\label{S1-009}
\iota_u(0) = \max_{x\in B} \iota_u(x).
\end{equation}
In particular, the following formula holds: 
\begin{equation}
\label{S1-010}
\iota_u(0) = \frac{\nu_{\hat u} (0)}{n} = \lim_{t\rightarrow -\infty} \frac{2t}{\log \left|    \{ u<t \}  \right|}.
\end{equation}
\end{theorem}

According to Kiselman's work on the integrability index  \cite{Kis}, 
Theorem (\ref{thm-main2}) immediately  implies that the estimate (equation (\ref{S1-001})) we obtained in Theorem (\ref{thm-main}) is sharp.

It is enough to prove the complex singularity exponent $\cC_u(x)$ always reaches its minimum at the origin, i.e. 
\begin{equation} 
\label{S1-0101}
\cC_u(0) = \min_{x\in B} \cC_u(x).  
\end{equation}
Notice that the symmetrization $\hat u$ is a radially symmetric, plurisubharmonic function with a single pole at the origin, 
and then it is well known (Proposition (\ref{prop-app-002}), Appendix) that we have 
\begin{equation}
\label{S1-0102}
\nu_{\hat u}(0) = n \iota_{\hat u}(0) = n   \iota_{\hat u}(B),
\end{equation}
for any such function. 
Therefore, the formula (equation (\ref{S1-010})) is obtained as in Proposition (\ref{prop-S1-001}): 
\begin{equation}
\label{S1-0103}
\nu_{\hat u}(0) = n \iota_{\hat u}(B) = n \iota_{u}(B) = n\iota_u (0).
\end{equation}

We begin with a lemma from Euclidean geometry. 
The proof is elementary, and we recall it for the convenience of the reader. 

\begin{lemma}
\label{lem-S1-004}
Let $z, z', s,   r, R$ be chosen as in Lemma (\ref{lem-S1-002}). 
For any measurable set $A\subset B_R(z)$, the rescaled set 
$ A' = s\cdot A $ will be contained in the ball $B_r(z')$. 
Moreover, we have 
$$ |A'| = s^{2n} |A|. $$
\end{lemma}
\begin{proof}
For any point $\z\in A'$, we can write $\z = s\cdot \xi $ for some $\xi\in A$. 
After identifying $\bC^n$ with $\bR^{2n}$, 
the two vectors $X_{\z}, X_{z'} \in \bR^{2n}$
will span a plain $\mathfrak{p}$ passing through the origin,
and  the two points  $z, \xi$ are also in this plane.  
Notice that the triangle  built by the three points $\{\underline{O}, \z', \z \}$ 
are similar to the triangle built by $\{ \underline{O}, z, \xi \}$. 
Therefore we have $| \z - z' | = s |\xi -z| < sR$, and hence $A'\subset B_r(z')$ by definition.

Next, the volume of a set $E$ can be taken as its $2n$-dimensional Hausdroff measure $\mathcal{H}^{2n}(E)$.  
Suppose the set $A$ is covered by a union of small open balls, i.e. $A \subset \bigcup_{j} B_{\ep} (x_j)$.
From what we just proved, 
the set  $A'$ will be covered by the union of their rescalings as 
$$ A' \subset \bigcup_j s\cdot B_{\ep}(x_j),$$
and the volume of each ball is rescaled by a factor $s^{2n}$. 
Hence we have $\mathcal{H}^{2n} (A') \leq s^{2n} \mathcal{H}^{2n} (A)  $,
and the reversed inequality follows from a similar argument. 
\end{proof}

According to Demailly-Koll\'ar \cite{DK}, there is another way to describe the complex singularity exponent. 
Let $u$ be a plurisubharmonic function on $\Omega$. 
For any point $x\in\Omega$,
we can consider the following set in $\bR$: 
$$\cE_u(x) = \left\{  c\geq 0;\  e^{-2ct}|\{ u < t \} | \ \emph{is bounded as $t\rightarrow -\infty$  for some $U\ni x$}       \right\}.$$
Then it is easy to see the following fact: 
$$ \cC_u(x) = \sup_{  \cE_u(x) } c. $$
By the famous openness conjecture (\cite{DK}, \cite{Bo1}),
we even have $\cC_u(x)\notin \cE_u(x)$.
In other words, any real number $c>0$ is no less than  $\cC_u(x) $ if, and only if 
for any $R>0$ small enough, and any  $k\in \mathbb{Z}^+$,
there exist a $t<0$ (depends on $R$ and $k$) such that we have 
\begin{equation}
\label{S1-011}
e^{-2ct} \left|\{ u < t \} \bigcap B_{R}(x) \right| > k. 
\end{equation}

%there exists a small $R>0$, a constant $C$ such that we have 
%$$ e^{-2ct} \left|\{ u < t \} \bigcap B_{R}(x) \right| < C, $$
%for $t<0$ small enough.  

Bearing this in mind,
under the assumption of Theorem (\ref{thm-main2}), we can argue as in Lemma (\ref{lem-S1-003}).  
Recall that we denote $\mathfrak{l}_z = \{ s\cdot z \},  s\in [0,1]$ by the line segment connecting the origin and a point $z\in B$. 

\begin{lemma}
\label{lem-S1-005}
The complex singularity exponent $\cC_u(x)$ is non-decreasing along $\mathfrak{l}_z$ possibly except at the origin. 
\end{lemma}
\begin{proof}
It is enough to prove  $ \cC_u(z') \leq \cC_u(z)$ for all $z' =sz, s\in (0,1)$.
For this purpose,  we can always assume $\cC_u (z) = c \in (0, +\infty)$.
Take any $k\in\mathbb{Z}^+$ and $R$ small enough 
such that equation (\ref{S1-011}) holds for some $t< 0$.  
Denote $A$ by the set 
$$A: = \{ u < t \} \bigcap B_{R}(z).  $$
For any fixed $s\in(0,1)$, the radius $r = sR$ 
will be smaller than $r_0$ ( equation(\ref{r0}))
when $R$ is small. 
Then the rescaled set $A' = sA$ is contained in the ball $B_r(z')$ by Lemma (\ref{lem-S1-004}). 

Writing $\z = s\cdot \xi$ for any point $\z \in A'$, we have $u(\z) \leq u(\xi)$. 
This is again because $u$ is non-decreasing along the line segment $\mathfrak{l}_{\xi}$
by Lemma (\ref{lem-S1-002}). Hence we have 
$$ A' \subset \{ u < t \} \bigcap B_{r}(z'). $$
Therefore, the following estimate is true in the ball $B_r(z')$
\begin{equation}
\label{S1-012}
 e^{-2ct} |  \{ u < t \} \bigcap B_{r}(z')    | \geq e^{-2ct} |A'| =  s^{2n} e^{-2ct} |A| > s^{2n} k.
\end{equation}
This implies $c\notin \cE_u(z')$, and then $ \cC_u(z') \leq \cC_u(z)$ follows.

\end{proof}

\begin{proof}[Proof of Theorem (\ref{thm-main2})]
It is left to prove equation (\ref{S1-0101}).
Suppose a point $z\in B$ is contained in the polar set of $u$, 
and we can assume $\cC_u(z) =c \in (0, +\infty)$ as before. 
Then we claim that the punctured disk $D_z^*$ is contained in the set $\{\cC_u (x) \leq c \}$. 
Since the complex singularity exponent is lower semi-continuous w.r.t the holomorphic Zariski topology \cite{DK},
the whole disk $D_z$ must be contained in the same set, and we conclude our proof. 

Based on Lemma (\ref{lem-S1-005}), it is again enough to prove that the complex singularity exponent 
is invariant on the boundary circle $\d D_z$.
This fact is true because the distance function and the measure are also invariant under the $S^1$-action, i.e.  for all $\theta\in\bR$, we have 
$$ |z - w| = | e^{i\theta}(z-w)|, $$
and 
$$ |A| = |e^{i\theta} A|,$$
for all $z, w\in B$ and any measurable subset $A\subset B$.
Then this invariance result follows from a similar argument as in Proposition (\ref{prop-S1-002}) and Lemma (\ref{lem-S1-005}).

\end{proof}

\begin{rem}
Besides the unit ball, our previous arguments also work on any other balanced domains in $\bC^n$. 
Therefore, for any $S^1$-invariant plurisubharmonic function $u$ on a balanced domain $\Omega$, 
its integrability index $\iota_u(x)$ always obtains its maximum at the origin, and formula (\ref{S1-010})
holds. 
\end{rem}

\section{ Toric plurisubharmonic functions}
In this section, we would like to study the \emph{residue Monge-Amp\`ere mass}
$$ \tau_u (0) = (dd^c u)^n|_{\{ \underline{O}\}} $$
of  a plurisubharmonic function $u$ at the origin $\underline{O}$. 
However, this quantity is not always well defined as we can see from Cegrell's example \cite{Ceg1}.
Even if it is well defined, there are only few ways to handle the complex Monge-Amp\`ere measure under the Schwarz symmetrization. 
Therefore, we will investigate plurisubharmonic functions with stronger symmetry than $S^1$-invariant at this stage.

A domain $\Omega\subset \bC^n$ is called a \emph{Reinhardt domain} if it is invariant under the following $(S^1)^{\times n}$-action: 
$$z \rightarrow  (e^{i\theta_1}z_1, \cdots, e^{i\theta_n } z_n), $$
for all $\theta_j \in \bR, j =1,\cdots, n$. 
A function $f$ defined on a Reinhardt domain $\Omega$ is called \emph{ toric }  if it satisfies 
$$f(e^{i\theta_1}z_1, \cdots, e^{\theta_n} z_n) = f(z)$$ 
for all $\theta_j \in \bR, j =1,\cdots, n$ at every $z\in \Omega$.

Let $\Omega\subset\bC^n $ be a bounded balanced Reinhardt domain, and $u$ be a toric plurisubharmonic function on it. 
As before, we assume that the polar set $\{ u = -\infty \}$ is non-empty and $\sup_{\Omega} u = \sup_{\d \Omega} u = 0$. 
Furthermore, we also assume that the function $u$ has only a single pole at the origin,
and then  its Monge-Amp\`ere measure $(dd^c u)^n$
 is well defined in terms of of the Bedford-Talyor-Demailly product \cite{Dem3}. 
 In particular, its residue Monge-Amp\`ere mass $\tau_u(0)$ is well defined. 

Now its symmetrization $\hat u: \hat\Omega \rightarrow \bR$ is a radially symmetric, non-decreasing
plurisubharmonic function, with only a single pole at the origin. 
It is well known that its residue mass is well defined and we further have 
\begin{equation}
\label{tor-0001}
\tau_{\hat u} (0) =  [\nu_{\hat u} (0)]^n,
\end{equation}
for such function as $\hat u$ (Proposition (\ref{prop-app-001}), Appendix).  
Then our previous identity (equation (\ref{S1-010}))  implies the following domination phenomenon of residue masses. 

\begin{theorem}
\label{thm-main3}
Let $u$ be a toric plurisubharmonic function on the unit ball $B$ with a single pole at the origin, 
which can be extended invariantly to a slightly larger ball $B_{1+\delta}$. Let $\hat u$ be its Schwarz symmetrization. 
Then we have 
\begin{equation}
\label{tor-0002}
\tau_{\hat u} (0) \leq  \tau_{ u} (0). 
\end{equation}
\end{theorem}
\begin{proof}
From our identity (\ref{S1-010}) and equation (\ref{tor-0001}), we obtain 
\begin{equation}
\label{tor-0003}
\tau_{\hat u} (0) =  n^n [\iota_u(0)]^n. 
\end{equation}
However, Kiselman \cite{Kis} proved the following identity for all toric plurisubharmonic functions 
$$ \iota_u (0) = \sup \{   \nu_u(0, a); \ \ a\in \bR_+^n, \ \sum_{j=1}^n a_j =1  \},$$
 where $\nu_u(0,a)$ is the refined Lelong number (see equation (\ref{intro-0040})) of $u$  at the origin
 in the direction 
 $$a= (a_1,\cdots, a_n), \ \ \forall a_j >0. $$ 

On the other hand, Rashkovskii \cite{Rash} proved a lower bound of the residue mass 
for all plurisubharmonic functions with a single pole at the origin as 
$$ \tau_u(0) \geq \frac{ [ \nu_u(0,a) ]^n }{ a_1\cdots a_n}, \ \ \ \forall a\in \bR_+^n. $$

Combining equation (\ref{tor-0003}), Kiselman's identity and  Rashkovskii's estimate, it is enough to prove that 
for all refined Lelong number $\nu_u(0,a)$ where $a\in \bR^n_+$ and  $\sum_{j=1}^n a_j =1$, we have 
\begin{equation}
\label{tor-0004}
 n\nu_u(0,a) \leq \frac{\nu_u(0,a)}{ (a_1\cdots a_n )^{\frac{1}{n}}}, 
\end{equation}
but this follows from the  inequality of arithmetic and geometric means 
$$ n ( a_1 \cdots a_n )^{\frac{1}{n}} \leq (a_1 + \cdots + a_n) = 1, $$
and our result follows.

\end{proof}

\begin{rem}
\label{rem-tor-1}
Again, our assumption on the domain is just for simplicity. 
This domination phenomenon for the residue masses under symmetrization 
occurs for all toric plurisubharmonic functions with a single pole at the origin, 
defined on any balanced Reinhardt domain $\Omega\subset \bC^n$.
\end{rem}

Next we will  give some examples of toric plurisubharmonic functions on the unit ball $B$ in $\bC^2$.
First, the following example shows that the estimate we obtained in Theorem (\ref{thm-main}) is sharp.  

\begin{example}
\label{ex-tor-001}
Consider the following function 
$$u(z) = \log |z_1| $$ 
defined on $B\subset\bC^2$.
It is clear that the Lelong number of $u$ at the origin is equal to $1$.
However, the sub-level set of $u$ is a ``complex cylinder" in the unit ball, i.e. 
$ \{ u < \log R \} = \{z\in B; \ \ |z_1| < R \}$, and then we have its volume 
$$ | \{ u < \log R \} | = \pi^2 (R^2 - R^4/2). $$
Therefore, the Lelong number of its symmetrization $\hat u$ at the origin is $2$,
since we have 
$$ \nu_{\hat u} (0) = \lim_{R\rightarrow 0} \ \frac{ 4 \log R}{\log (R^2 - R^4/2) + \log 2}=2,  $$
and this implies $\iota_{ u} (0 ) =1$ and $\tau_{\hat u}(0) = 4$.

\end{example}

If the function $u$ is already radially symmetric, then its Schwarz symmetrization is itself. 
Therefore, we have $\nu_u(0) = \nu_{\hat u} (0)$ in this case.
Next, we provide an example where the value $\nu_{\hat u}(0)$ is in between.

\begin{example}
\label{ex-tor-002}
Consider the following function on $B$
$$ u(z) = \log (|z_1|^2 + |z_2|^{\frac{1}{2}} ). $$
The Lelong number of $u$ at the origin is equal to $\frac{1}{2}$, since we have 
\begin{eqnarray}
\label{tor-001}
 \log (|z_1|^2 + |z_2|^{\frac{1}{2}} ) - \log 2 &\leq& \max  \{\log |z_1|^2 , \log |z_2|^{\frac{1}{2}} \}
 \nonumber\\
 &\leq & \log (|z_1|^2 + |z_2|^{\frac{1}{2}} ),
 \end{eqnarray}
 and the well known equation $\nu_{\max\{u,v \}} (x) = \min\{\nu_u(x), \nu_v(x) \}$ for two plurisubharmonic functions $u, v$ \cite{abook}. 
 By Demailly's comparison theorem \cite{Dem3}, it follows from equation (\ref{tor-001}) that we have $\tau_{u}(0) = 1$. 
 On the other hand, the sub-level set of $u$ is an ellipsoid: 
 $$\{ u < 2\log R\} = \{  |z_1|^2 + |z_2|^{\frac{1}{2}} < R^2 \}, $$
 and its volume can be computed as 
 \begin{eqnarray}
 \label{tor-002}
| \{ u < 2\log R\} | &=& 4\pi^2 \int_0^R r_1 dr_1 \int_0^{(R^2 - r_1^2)^2} r_2 dr_2
\nonumber\\
&=& 2\pi^2 \int_0^R (R^2 - r_1^2)^4 r_1 dr_1 
\nonumber\\
&=& O(R^{10}).
 \end{eqnarray}
 Therefore, we have 
 $$\nu_{\hat u} (0) = \lim_{R\rightarrow 0} \frac{ 2\log R}{ \frac{1}{4} ( \log R^{10} + O(1) ) } =  \frac{4}{5}, $$ 
 and this implies  $\iota_u(0) = \frac{2}{5}$ and $\tau_{\hat u} (0) = \frac{16}{25} < \tau_u (0)$.

\end{example}

In general, Demailly \cite{Dem3} considered the following function on $B$ for any $0< \ep < 1 $ as 
$$ u(z) = \max \{ \ep^{-1} \log|z_1|, \ep \log |z_2| \}.$$
One can show that its residue mass is always $1$ at the origin, whereas its Lelong number at the origin is $\ep$.
In this case, we have for its symmetrization $\nu_{\hat u}(0) =  2 (\ep + \ep^{-1})^{-1}$, $ \iota_{u}(0) = ( \ep + \ep^{-1})^{-1}$ and 
$\tau_{\hat u}(0) = 4(\ep + \ep^{-1})^{-2} < 1$.

The next example was provided by Kiselman \cite{Kis1}, 
and we can see that the Monge-Amp\`ere measure near the origin is indeed ``regularised "
by the Schwarz symmetrization. 

\begin{example}
\label{ex-tor-003}
Consider the following function on $B_{1/2}$
$$ u (z) = (-\log |z_1|)^{\frac{1}{2}} (|z_2|^2 - 1). $$
This function is smooth outside the hyperplane $H = \{z_1 =0 \}$, 
and its Monge-Amp\`ere measure is 
$$ \det(u_{j\bar k}) = \frac{ 1 - 2 |z_2|^2}{ 8n |z_1|^2 (-\log |z_1)},   $$
on $B_{1/2} \backslash H$.
This measure will accumulate infinite mass near any point on $H$, 
and we can say $\tau_u(0) = +\infty$. 

However, it is easy to see that the Lelong number of $u$ is zero everywhere, 
and hence the integrability index $\iota_u(x)$ is also zero for all $x\in B_{1/2}$.
Therefore, we have 
$ \tau_{\hat u}(0) = 4 [\iota_u(0)]^2 = 0$.

\end{example}

Finally, we present Cegrell's example \cite{Ceg1}, for which the residue mass is not uniquely determined by decreasing sequences. 
\begin{example}
\label{ex-tor-004}
Consider the function on $B$ 
$$ u(z) =2 \log |z_1 z_2|.  $$
It is easy to see that the Lelong number $\nu_u(0) = 4$.
However, its residue Monge-Amp\`ere mass can not be well defined at the origin. 
In fact, the the following two smooth sequences $\{u_j \}, \{v_j \}$ are both decreasing to $u$ 
$$ u_j = \log ( |z_1 z_2|^2 + 1/j ), $$
and 
$$v_j =\log(|z_1|^2 + 1/j) + \log(|z_2|^2 +1/j).$$
Then one can show that $(dd^c u_j)^2$ is zero for every $j$, whereas
$(dd^c v_j)^2$ converges weakly to $32\delta_0$, where $\delta_0$ is the Dirac mass at the origin.

On the other hand, it is easy to see that its integrability index at the origin is equal to $2$ since we have 
$$ e^{-2c u} =   |z_1|^{-4c} |z_2|^{-4c}.$$
Hence it follows that we have  $\nu_{\hat u}(0) = 4$ and $\tau_{\hat u} (0)  = 16 $.
\end{example}

For any $S^1$-invariant plurisubharmonic function $u$ with a single pole at the origin, its residue mass is always well defined.  
Then we can still ask a similar question about this domination phenomenon. 

\begin{conjecture}
Let $u$ be an $S^1$-invariant plurisubharmonic function with a single pole at the origin on a balanced domain $\Omega$, 
and $\hat u$ be its symmetrization. Then we have 
$$\tau_{\hat u}(0) \leq \tau_{u} (0). $$
\end{conjecture}

\section{Appendix}
Let $\cR$ be the class of all radial, upper semi-continuous  functions on the unit ball $B\subset \bC^n$.
Denote $PSH^{\infty}(B)$ by the family of plurisubharmonic functions on $B$ with non-empty polar set. 
By the maximum principle, 
any $u\in \cR \bigcap PSH^{\infty}(B)$ is non-decreasing and has only a single pole at the origin. 

For such function, the measure $(dd^c u)^n$ is well defined \cite{Dem3}, 
and then we have the following relation between the residue mass
and the Lelong number at the origin.

\begin{prop}
\label{prop-app-001}
For any $u\in \cR \bigcap PSH^{\infty}(B)$, we have $\tau_u(0) = [ \nu_u (0)]^n$. 
\end{prop}

Before going to the proof, the following regularization technique is standard. 
\begin{lemma}
\label{lem-app-001}
For any $u\in \cR \bigcap PSH^{\infty}(B)$, there exits a sequence of smooth plurisubharmonic functions $u_j \in \cR $
decreasing to $u$. In particular, we have 
\begin{equation}
\label{app-001}
\int_K (dd^c u)^n = \lim_{j\rightarrow +\infty}  \int_K (dd^c u_j)^n,
\end{equation}
for any relative compact Borel subset $K$ of $B$.
\end{lemma}
\begin{proof}
Let $\rho(x) $ be the standard cut-off function on $\bC^{n}$. 
That is to say, a smooth function $\rho$ is supported on the unit ball of $\bC^{n}$ 
with $\rho(x) = \rho(|x|)$ and $\int_{\bC^{n}} \rho(x) dx = 1$.
Denote $\rho_j$ by its rescaling as 
$ \rho_j (x)  = j^{2n} \rho \left( jx \right) $.
Consider the following regularization 
$$ u_j (x) = \int_{\bC^n} u (y-x) \rho_j (y) d\lambda(y), $$
and then $u_j$ is a sequence of smooth plurisubharmonic functions decreasing to $u$. 
Moreover, we claim that $u_j$ is radial.  

Let $x$ be a point in $B$.
Any other point $w$ which is differ from $x$ by a rotation can be written as 
$w = A\cdot x$, for some special orthogonal matrix $A$. 
Then we have for any $y\in \bC^n$ and $z = A^{-1}\cdot y $
\begin{eqnarray}
\label{app-002}
u_j(w) &=& 
\int_{\bC^n} u( y - Ax) \rho_j (y) d\lambda(y)
\nonumber\\
&=& \int_{\bC^n} u ( A^{-1}\cdot y - x ) \rho_j (y) d\lambda(y)
\nonumber\\
&= & \int_{\bC^n} u( z - x) \rho_j (z) d\lambda(z) 
\nonumber\\
&=& u_j(x).
\end{eqnarray}
The identity on the third line of equation (\ref{app-002}) is because 
the cut off function $\rho_j$ and the Lebesgue measure $d\lambda$ are both invariant under the action by $A$,
and our result follows. 

\end{proof}

\begin{proof}[Proof of Proposition (\ref{prop-app-001})]
By Lemma (\ref{lem-app-001}), it boils down to prove for any smooth radial plurisubharmonic function $u$ on $B$ we have 
\begin{equation}
\label{app-003}
\int_{B_R} (dd^c u)^n = \left\{ \frac{1}{a_{2n-2} R^{2n-2}}\int_{B_R} \frac{\Delta u}{ 2\pi} \right\}^n,
\end{equation}
for any small radius $R>0$.

Writing $|z| =r$ and $t = \log r$, the function $f(t) = y(r) = u(z)$ is convex for $t<0$.
Thanks to Theorem (2.32) in \cite{GZ},  
the RHS of equation (\ref{app-003}) is exactly equal to 
$$   [ R\cdot  \d_r y (R) ]^n  =  [ \d_t f(T) ]^n, $$
for almost everywhere $R\in [0, 1)$ and $T =\log R$.
 
On the other hand, we have 
$$ (dd^c u)^n = \frac{2^n n!}{\pi^n} \det (u_{j\bar k}) d\lambda, $$
and then an easy computation shows 
\begin{eqnarray}
\label{app-004}
\det(u_{j\bar k}) &=& \frac{1}{2^{n+1}} \left( y'' + \frac{y'}{r} \right) \left( \frac{y'}{r} \right)^{n-1}
\nonumber\\
&=& \frac{1}{n 2^{n+1}} \frac{1}{r^{2n-1}} \{ (ry')^n \}'.
\end{eqnarray}
Using the spherical coordinate, 
 the LHS of equation (\ref{app-003}) is equal to the following 
\begin{eqnarray}
\label{app-005}
\int_{B_R} (dd^c u)^n  &=&  2\pi a_{2n-2}\int_0^R \frac{(n-1)!}{2\pi^n} \det(u_{j\bar k})r^{2n-1}dr
\nonumber\\
&= & \int_0^R \{ (ry')^n \}' dr
\nonumber\\
&=&   \left[ R \frac{\d y}{\d r} (R) \right]^n,
\end{eqnarray}
since $a_{2n-2} = \frac{\pi^{n-1}}{ (n-1)!}$, and we conclude the proof.

\end{proof}

For radial plurisubharmonic functions, there is also a simple relation between the Lelong number and  the integrability index at the origin. 
\begin{prop}
\label{prop-app-002}
For $u\in \cR \bigcap PSH^{\infty}(B)$, we have $ \nu_u(0) =  n\iota_u(0) $. 
\end{prop} 
\begin{proof}
By the estimate (equation (\ref{intro-000})), it is enough to prove for any $ 0< c  < n\nu^{-1}_u(0)$, the integral 
$$ \int_{B_R} e^{-2c u }  d\lambda$$
is finite for $R>0$ small enough. Writing $f(t) = u (z)$ and $T = \log R$ as usual, the convex function $f(t)$ is Lipschitz continuous,
and then we can apply the fundamental theorem of Calculus as 
\begin{eqnarray}
\label{app-006}
&&\frac{1}{2\pi a_{2n-2}} \int_{B_R} e^{-2c u} d\lambda = \int_{-\infty}^T e^{-2c f} e^{2nt} dt
\nonumber\\
&=& e^ {2nt  \left( 1- \frac{cf}{nt}  \right) } \big|_{- \infty}^{T} + \frac{c}{n} \int_{-\infty}^T f'(t) e^{ -2cf + 2nt } dt.
\end{eqnarray}
Notice that the two positive functions $f'(t)$ and $t^{-1}f(t)$ are non-decreasing, and both converge to $\nu_u(0)$ as $t\rightarrow -\infty$.
By  our assumption on $c$,
there exists some $t_0 < 0$ such that we have
$$  \min\{ 1 - \frac{c}{n} f'(t), 1-\frac{c}{nt} f(t)  \} > \ep,$$
for some $\ep >0$ small and all $t< t_0$.
Picking up $T  =  t_0 -1 $, we have 
\begin{eqnarray}
\label{app-007}
e^ {2nt  \left( 1- \frac{cf}{nt}  \right) } \big|_{- \infty}^{T} &=&   \int_{-\infty}^T  ( 1-  \frac{c}{n}f'(t) ) e^{ 2nt \left( 1 - \frac{cf}{nt} \right) } dt
\nonumber\\
&\geq&  \ep \int_{-\infty}^T  e^{ -2cf + 2nt} dt.
\end{eqnarray}
Since $ \ep <  1- \frac{cf}{nt}  < 1$ for all $t\in (-\infty, T)$, the LHS of equation (\ref{app-007}) is bounded above by $e^{2nT}$,
and our result follows.

\end{proof}

\end{document}